\documentclass[letterpaper, 10 pt, conference]{IEEEtran}

\usepackage{amssymb,amsfonts,amsmath,amsthm}
\usepackage{mathtools}
\usepackage{eucal,bm}
\usepackage{gensymb,subfigure}
\usepackage{multirow}
\usepackage{booktabs}
\usepackage{times,color,url}
\usepackage{comment}
\usepackage{cite}
\usepackage{graphicx}
\usepackage[all]{xy}
\usepackage{algorithm}
\usepackage{algpseudocode}
\usepackage{subfigure}
\usepackage{bm}
\usepackage{color}
\usepackage{wrapfig}
\usepackage{epsf}
\usepackage{times,mathptm}
\usepackage{url}
\usepackage{nomencl}
\newtheorem{theorem}{Theorem}

\newtheorem{lemma}{Lemma}

\newcommand{\squeezeup}{\vspace{-2.5mm}}
\begin{document}
\title{Tractable Structure Learning in Radial Physical Flow Networks}
\author{\IEEEauthorblockN{Deepjyoti~Deka*, Scott~Backhaus*, and Michael~Chertkov*\\}
\IEEEauthorblockA{*Los Alamos National Laboratory, USA\\
Email: deepjyoti@lanl.gov, backhaus@lanl.gov, chertkov@lanl.gov}}

\maketitle

\begin{abstract}
Physical Flow Networks are different infrastructure networks that allow the flow of physical commodities through edges between its constituent nodes. These include power grid, natural gas transmission network, water pipelines etc. In such networks, the flow on each edge is characterized by a function of the nodal potentials on either side of the edge. Further the net flow in and out of each node is conserved. Learning the structure and state of physical networks is necessary for optimal control as well as to quantify its privacy needs. We consider radial flow networks and study the problem of learning the operational network from a loopy graph of candidate edges using statistics of nodal potentials. Based on the monotonic properties of the flow functions, the key result in this paper shows that if variance of the difference of nodal potentials is used to weight candidate edges, the operational edges form the minimum spanning tree in the loopy graph. Under realistic conditions on the statistics of  nodal injection (consumption or production), we provide a greedy structure learning algorithm with quasi-linear computational complexity in the number of candidate edges in the network. Our learning framework is very general due to two significant attributes. First it is independent of the specific marginal distributions of nodal potentials and only uses order properties in their second moments. Second, the learning algorithm is agnostic to exact flow functions that relate edge flows to corresponding potential differences and is applicable for a broad class of networks with monotonic flow functions. We demonstrate the efficacy of our work through realistic simulations on diverse physical flow networks and discuss possible extensions of our work to other regimes.
\end{abstract}

\section{Introduction}
\label{sec:intro}
Physical flow networks \cite{dembo1989or} form strategic components of modern society's activities and help in the mass transport of energy and daily utilities from far off-generation points to end users through pipes/edges. Example of such networks include the power grid (flow of electricity), gas transmission network (delivery of gas for household usage and energy production), water pipelines (for drinking and industrial water) etc. Further, in the continuum domain (fluid approximation), discrete particular network like vehicular traffic can also be considered as a physical flow network. Over time, such networks have grown in size and become vital for the smooth functioning of most, if not all, activities - be it industry or household related. Efficient control and fast monitoring of the state of such flow networks is crucial for their real-time usage. Further, the advancement in smart active devices (energy meters, smart thermostats/heaters etc.) has led to efforts in distributed but optimal control of flow networks. Fast topology (set of inconnections in the flow network) and usage (steady state of current vehicular traffic, statistics of power, gas, water consumption etc.) estimation are necessary tools to ensure the optimal application of decentralized schemes. Finally the assimilation of online communication tools for monitoring and signaling exchange in flow networks puts them at risk of cyber adversaries and disruptive agents. Low overhead algorithms to learn the topology and state of the network can help quantify the cyber risk associated with compromised set of measurements and help guide preventive measures and placement of secure devices.

Due to the diversity of transported commodities, flow networks vary in the analytical flow models. However such flow models often satisfy common laws. One, net flow at each node is conserved, i.e., the total flow injected into each node is equal to the sum outward flow on all edges connected to that node. Second, the flow on each edge is guided by the difference between potentials at the two nodes on either side of the edge. Examples of such potential include voltages in power networks, pressure in gas and water networks. Similarly, one can think of virtual potentials in traffic network fluid models. In most flow networks, the potential difference across an edge is expressed as a monotonic function of the flow on it, implying that the flow increases when the difference in potentials increases and vice versa. Note that the monotonic function may be non-linear. In this paper, we discuss topology learning algorithms for radial flow networks with monotonic flow functions. We focus specifically on radially structured networks as they appear frequently in different contexts. For example, electricity distribution grids \cite{hoffman2006practical} are known to operate in a radial topology due to economic and operational reasons \cite{distgridpart1}. Similarly, gas transmission networks \cite{gashandbook,misra2015optimal} and some water networks \cite{aminwater} have a tree-like operational configuration. However, the set of all permissible edges in the network (operational and open) form an underlying loopy graph with cycles. The radial operational configuration is achieved by restricting the flow to a subset of the permissible edges in the network as shown in Fig.~\ref{fig:city}. In certain networks like the electricity distribution grid, this radial configuration can be changed over a few hours by switching on and off edges (transmission line breakers) and needs to be estimated for control applications. Real-time meters on edges that relay information on current flow and operational status are often sparsely present. Even if the radial structure is static, third-party applications may be interested in learning the topology using indirect (non-edge based) measurements as access to network structure information is often restricted. Hence, we analyze the problem of estimating the true operational structure using only empirical nodal potential measurements. Specifically, \textit{we show that second moments of nodal potentials are sufficient to reconstruct the operational radial topology by a greedy algorithm}. Surprising, our learning scheme does not depend on the exact flow function (linear or non-linear) for the network as long as the flow function is monotonic. Thus, it has wide applicability for a variety of networks as mentioned in detail in subsequent sections. Note that brute force approaches to learn the topology is computationally prohibitive due to the exponential number of radial topologies that can be constructed from a dense loopy graph of permissible edges.

\subsection{Prior Work}
Past work in learning the structure of flow networks have generally focussed on specific applications. In particular, nodal measurement based structure estimation of power distribution networks is an area of active research.  Researchers have used Markov random fields \cite{he2011dependency}, signs in inverse covariance matrix of voltages (potentials) \cite{bolognani2013identification}, envelope comparison based reconstruction methods \cite{berkeley,sandia1,sandia2} etc. Limited power flow measurements have been used to estimate the topology using maximum likelihood tests in \cite{ramstanford}. In our prior work \cite{distgridpart1,distgridpart2,distgrid_ecc}, iterative greedy learning schemes based on a linear power flow model have been used to determine the operating power grid, even in the presence of missing/unobserved nodes.

We are not aware of prior work in learning the structure of gas or water networks using information of nodal potentials. However there exist several efforts on different optimization (stochastic, robust etc.) and control (optimal, distributed etc.) schemes for these networks that depend on information of network structure and use nodal potentials as variables. These include geometric programming based optimization schemes specifically for radial gas \cite{misra2015optimal} and water networks \cite{aminwater}. Recent work \cite{vuffray2015monotonicity} has demonstrated the tractability of several robust optimization schemes in gas networks due to the monotonic nature of the function that relates edge flows to nodal potential differences. The learning algorithm in this paper shows that the monotonicity of the flow function makes structure learning using nodal potentials tractable as well. These algorithms can in turn enable optimization problems to be tackled without the prior knowledge of the underlying flow network as that can be easily estimated.

Aside from the mentioned work in flow networks, the Chow-Liu algorithm \cite{chow1968approximating} uses a spanning tree algorithm for learning tree-structured graphical models that is based on the pairwise-factorization of the systems's mutual information. This is generalized in \cite{choi2011learning} to tree-structured graphical models with hidden/latent variables through the use of information distances as edge weights.

\subsection{Contribution of This Work}
Most of the previous work in learning structure of physical flow networks are limited to specific cases, in particular power grids. Further they assume linear flow function relating edge flow and nodal potentials. The fundamental contribution of this paper is to develop learning algorithms that are applicable for physical flow networks with monotonic flow functions that can be nonlinear and even distinct for each edge in the system. We show that under independent nodal injections, the variance of potential differences in such networks show provable trends that can guide greedy algorithms for structure learning. Our main algorithm uses variance in difference of nodal potentials as edge weight and identifies the operational structure by a spanning tree algorithm. In particular, the algorithm does not need any information of the flow function involved or nodal injection statistics. If the flow functions are known, the algorithm can be used to estimate the statistics of nodal injections. In essence, this work generalizes prior work \cite{distgrid_ecc} on linear power flow models to general radial networks with monotonic flow functions that are distinct and non-linear. The worst-case computational complexity of our algorithm is quasi-linear in the number of permissible edges in the network, which is efficient to learn the structure of large networks. We are not aware of any existing work that estimates the structure of general flow networks with non-linear flow functions. We demonstrate the performance of our algorithms through experiments on two test networks, one pertaining to power grid and the other to gas grid.

The rest of the manuscript is organized as follows. Section \ref{sec:structure} introduces nomenclature and relations between injections, flows and potentials in physical flow networks through detailed example networks. We present key properties and trends in nodal potentials for flow networks in Section \ref{sec:trends}. Design of the spanning tree based learning algorithm is given in Section \ref{sec:algo1}. We also include a part on extensions of our work in Section \ref{sec:algo1}. Simulation results of our learning algorithm on different example networks are presented in Section \ref{sec:experiments}. Finally, Section \ref{sec:conclusions} contains conclusions, extensions and discussion of future work.
\squeezeup
\section{Flow Models for Flow Networks}
\label{sec:structure}
We first provide the notation for the topology of the flow network.

\textbf{Radial Structure}: Mathematically, the overall physical flow network is represented as a loopy graph ${\cal G}=({\cal V},{\cal E})$, where ${\cal V}$ is the set of nodes and ${\cal E}$ is the set of all permissible edges. Nodes are denoted by alphabets ($a$,$b$,...) and edge between two nodes $a$, $b$ by node pair $(ab)$. The `radial' structure composed of operational edges is denoted by tree $\cal T$ with nodes ${\cal V}_{\cal T}$ and operational edge set ${\cal E}_{\cal T} \subset {\cal E}$. We restrict our discussion to one operational tree as shown in Fig.~\ref{fig:city} though our results hold for the case with multiple disjoint trees.
\begin{figure}[!bt]
\centering
\includegraphics[width=0.20\textwidth, height =.23\textwidth]{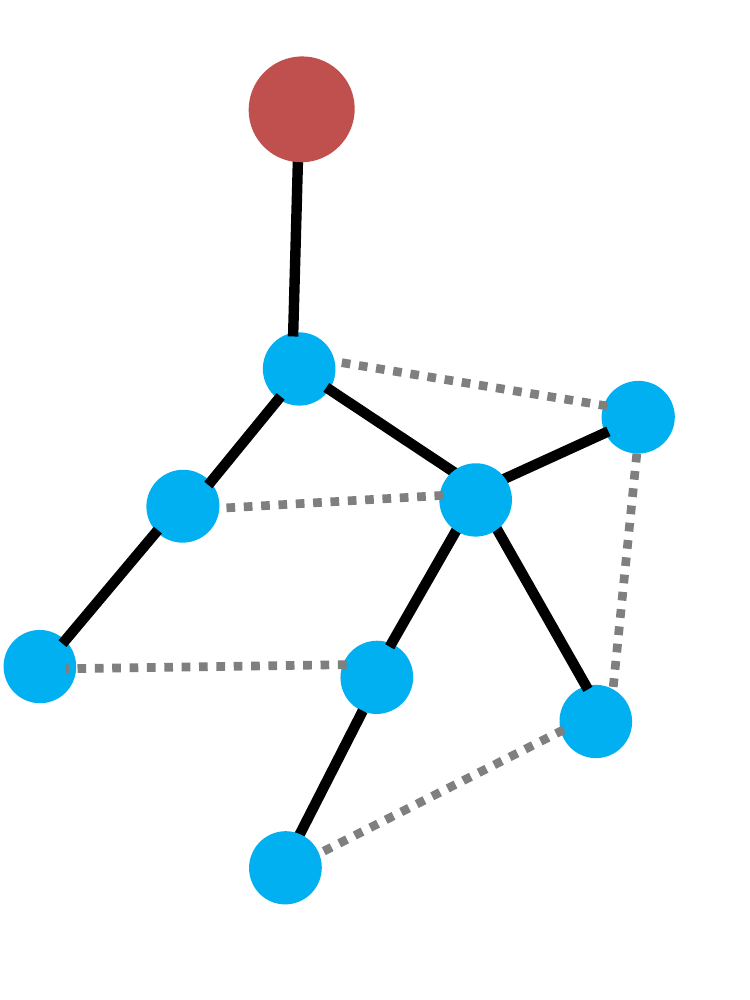}
\squeezeup
\squeezeup
\caption{A radial flow network with operational edges colored solid black. Dotted grey lines represent non-operational edges. The red node denotes the `reference' node.
\label{fig:city}}
\end{figure}
Next, we list the flow equations for the network.

\textbf{Flow Models}: The flow network is characterized by three sets of variables:  nodal injections, edge flows and nodal potentials. At each node $a$, flow is conserved and $P_a = \sum_{b:(ab)\in {\cal E}_T}f_{ab}$ where $P_a$ is the injection at $a$ and $f_{ab}$ is the flow from $a$ to $b$ on edge $(ab)$. In vector form, we write
\begin{align}
P = M^Tf \label{inj}
\end{align}
Here $M$ is the node to edge incidence matrix in tree $\cal T$. Each edge $(ab)$ in the network is represented by a row equal to $(e_a^T - e_b^T)$ in $M$. Here $e_a$ denotes the standard basis vector with $1$ at the $a^{th}$ position. Note that $\textbf{1}^TP =\textbf{1}^TPM^Tf =0$ where $\textbf{1}$ is the vector of all ones. Thus the network is `lossless' and total flow is conserved. Next, the flow $f_{(ab)}$ on edge $(ab)$ and potentials $\pi_a$ and $\pi_b$ at nodes $a$ and $b$ satisfy
\begin{align}
\pi_a-\pi_b = g_{ab}(f_{ab}) \label{flow}
\end{align}
where $g_{ab}$ is the monotonic flow function for edge $(ab)$ and can be distinct for each edge.  Further, the flows and injections are unchanged if all nodal potentials are increased/decreased by the same amount. Following standard practice \cite{abur2004power,misra2015optimal}, one node's potential can be considered as reference and potentials are measures relative to that of the reference node. The substation or node with largest production of power or gas is generally considered as the reference node. We give the following examples of lossless flow networks.

\textbf{Power Distribution Grid}: Distribution grid \cite{hoffman2006practical} is the final tier of the power grid that extends from the distribution substation to the end-users. Flows in the radial distribution grid are composed of active and reactive power flows that are related to nodal voltages according to Kirchoff's laws. During stable operations, the line flows can be expressed by the following lossless approximation commonly termed as LinDistFlow model \cite{89BWa,89BWb,89BWc}:
\begin{align}
&P_a = \sum_{\substack{(ab)\in{\cal E}^{{\cal T}}\\b \neq a}} f^p_{ab}, Q_a = \sum_{\substack{(ab)\in{\cal E}^{{\cal T}}\\b \neq a}} f^q_{ab}
\label{injpower}\\
&v_a^2 - v_b^2 = 2\left(r_{ab}f^p_{ab}+x_{ab} f^q_{ab}\right)  \forall (ab) \in {\cal E}^{{\cal T}}\label{flowpower}
\end{align}
Here $P_a$ ($Q_a$) is the nodal active (reactive) power injection at node $a$ while $v_a$ is the voltage magnitude. $f^p_{ab}$ ($f^q_{ab}$) is the active (reactive) flow on edge $(ab)$ and $r_{ab}$ ($x_{ab}$) is the line resistance (reactance). From Eq.~(\ref{flowpower}), it is clear that squares of voltage magnitudes ($v^2$) represent potentials here and the flow function that relates potentials to line flow is linear. Similar linear flow functions in this area include linear coupled (LC) AC power flow model \cite{distgridpart1,bolognani2016existence} with complex voltages as potentials and resistive DC power flow models \cite{abur2004power} with phase angles as potentials.

\textbf{Natural Gas Transmission Network}: In gas grids, natural gas is driven from generators to consumers (households, gas turbines) through pipelines \cite{gashandbook,misra2015optimal}. During steady state, gas flow is governed by the following relation:
\begin{align}
&P_a = \sum_{(ab)\in{\cal E}^{{\cal T}}} f_{ab}\label{injgas}\\
&\phi_a^2 - \phi_b^2 = \alpha_{ab}f_{ab}|f_{ab}|+ \beta_{ab} \forall (ab) \in {\cal E}^{{\cal T}}\label{flowgas}
\end{align}
Here $P_a$ denotes the nodal injection. $f_{ab}$ is the gas flux (flow per unit length) from node $a$ to $b$. $\phi_a$ is the pressure at node $a$ and its second power represents the potential. The quantity $\beta_{ab}$ is the pressure boost provided by the compressor on edge $(ab)$ and is constant over short time intervals where changes in flow are driven by difference in nodal potentials. Further, $\alpha_{ab}$ represents a constant factor that depends on friction, length and diameter of the pipe (edge $(ab)$) as well as temperature, universal gas constant and gas compressibility \cite{misra2015optimal}. Note that the flow function here is second-order but monotonic.

\textbf{Radial Water Network}: Water networks consists of pipes where nodal `head pressures' at their ends control the direction and quantity of flow in them \cite{aminwater,boulos2006comprehensive}. The flow equations are non-linear and similar to Eqs.~(\ref{injgas}),(\ref{flowgas}) for gas networks, but with a different exponent ( $>1$) for flow. We omit describing them mathematically for brevity.
Similarly, radial \textbf{traffic networks} also satisfy conservation of flow at each node and can be modelled in a similar way \cite{como2010robustness}.

We use $\mu_{g(X)}$ and $\Omega_{g(X)}$ to denote the mean and variance of function $g$ defined over random variable $X$. Similarly, $\Omega_{g(X)h(Y)}$ denotes the covariance (centered second moment) of functions $g$ and $h$ defined over random variables $X$ and $Y$ respectively. Here $X$ and $Y$ may be correlated. Thus
\begin{align}
&\mu_{g(X)} = \mathbb{E}[g(X)],~\Omega_{g(X)} = \mathbb{E}[(g(X)- \mu_{g(X)})^2],\nonumber\\
& \Omega_{g(X)h(Y)} = \mathbb{E}[(g(X)- \mu_{g(X)})(h(X)- \mu_{h(X)})]\label{covar}
\end{align}
In the next Section, we derive algebraic properties of second moments of nodal potentials in radial networks using the flow functions. These properties will help derive our learning algorithms.

\section{Trends in Second moment of Potentials in Radial Networks}
\label{sec:trends}
Let tree $\cal T$ denote the operational radial flow network with edge set ${\cal E}_{\cal T}$. Without a loss of generality, we assume that all edges are directed towards the reference node. We denote the unique path (sequence of edges) from any node $a$ to the reference node in tree ${\cal T}$ by ${\cal P}_a^{{\cal T}}$. The set of all nodes $b$ such that path ${\cal P}_{b}^{\cal T}$ passes through node $a$ is called the `descendant' set $D_{a}^{\cal T}$ of node $a$. By definition, $a \in D_{a}^{\cal T}$. If $b \in D_{a}^{\cal T}$ and $(ba) \in {\cal E}_{\cal T}$), we term $a$ as parent and $b$ as its child. See Fig.~\ref{fig:picinc} for an illustrative example.

Eqs.~(\ref{inj},\ref{flow}) represent the relation between injections ($P$), flows ($f$) and potentials ($\pi$) in the network. As stated in the previous section, the potential at the reference node is fixed, while its injection is given by negative sum of all other nodal injections (due to lossless property). The number of degrees of freedom in the injection or potential vector is thus one less than the number of nodes in the system. Following standard practice, we analyze the `reduced' network of only non-reference nodes with nodal potentials measured relative to that of the reference node. We remove the column corresponding to the reference node from the incidence matrix $M$ and omit its injection and potential terms from the vectors $P$ and $\pi$ respectively. Abusing notation, we use $M$, $P$ and $\pi$ to refer to the reduced versions of their respective definitions in the remaining part of this paper. Note that the reduced incidence matrix $M$ has full rank. As all edges are directed towards the reference node, the inverse $M^{-1}$ has the following specific analytical structure \cite{68Resh}(also see Fig.~\ref{fig:picinc}):
\squeezeup
\begin{align} 
\squeezeup
{\huge M}^{-1}(a,r)=\begin{cases}1 & \text{if edge $r\in {\cal P}_a^{{\cal T}}$}\\
0 & \text{if edge~} r \not\in {\cal P}_a^{{\cal T}} \end{cases} \label{treeinv}
\squeezeup
\end{align}
\squeezeup
\begin{figure}[ht]
\centering
\subfigure[]{\includegraphics[width=0.15\textwidth,height = .16\textwidth]{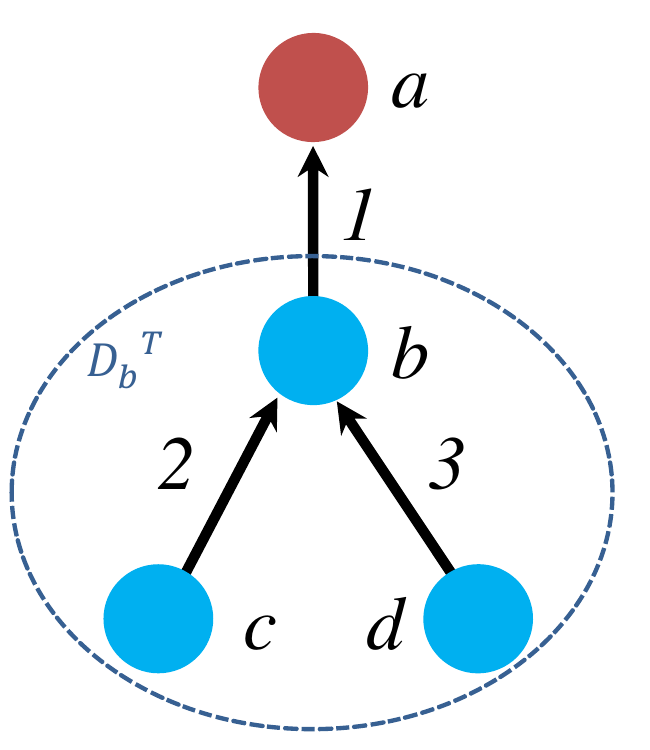}\label{fig:picinc3}}
\subfigure[]{\includegraphics[width=0.16\textwidth,height=0.08\textwidth]{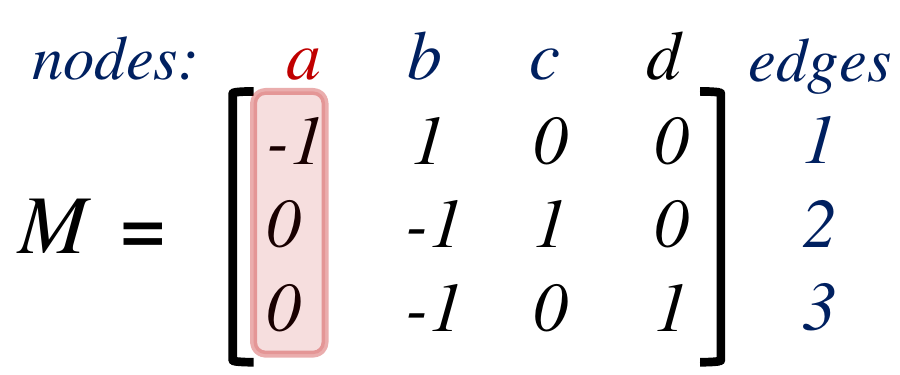}\label{fig:picinc_1}}
\subfigure[]{\includegraphics[width=0.15\textwidth,height=0.08\textwidth]{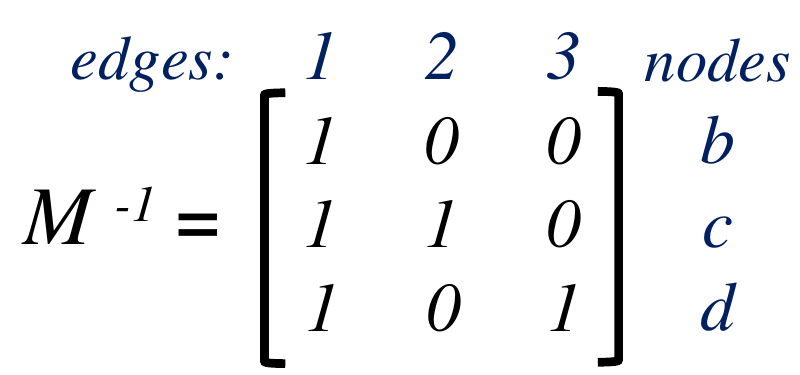}\label{fig:picinc1}}
\squeezeup
\caption{(a) Radial network with four nodes ($a,b,c,d$) and three edges ($1,2,3$) directed toward the reference node $a$. ${\cal P}_c^{{\cal T}} = \{(cb),(ba)\}$. $D_{b}^{\cal T}= \{b,c,d\}$. $b$ is the parent and $c,d$ are its children nodes. (b) Reduced incidence matrix $M$ derived by removing column corresponding to node $a$. (c) $M^{-1}$ as per Eq.~(\ref{treeinv}).
\label{fig:picinc}}
\vspace{-3mm}
\end{figure}

Using Eq.~\ref{inj} with Eq.~(\ref{treeinv}), the flow on edge $(ab)$ with node $a$ and its parent $b$ satisfies:
\squeezeup
\begin{align}
f_{ab} = \sum_{c \in D_a^{\cal T}} P_c \label{flowinv}
\end{align}
Observe the flow relation in Eq.~(\ref{flow}). Let $k_1, k_2, ..k_r$ be the sequence of $r$ intermediate nodes between a node $k$ and its descendant $a$. Using a telescopic sum for nodal potentials, we have
\squeezeup
$$\pi_k - \pi_a = \pi_k - \pi_{k_1} + \pi_{k_1} - \pi_{k_2}+...+ \pi_{k_r}-\pi_a
             =  \smashoperator[lr]{\sum_{(cd) \in {\cal P}_a^{{\cal T}} - {\cal P}_k^{{\cal T}}}}             g_{cd}(f_{cd})$$
where ${\cal P}_a^{{\cal T}} - {\cal P}_k^{{\cal T}}$ consists of edges that lie in path from node $a$ to $k$. For any two nodes $a$ and $b$, we can find some node $k$ on the path from $a$ to $b$ such that $a,b$ are both descendants of $k$ and ${\cal P}_b^{{\cal T}} \cap {\cal P}_a^{{\cal T}} = {\cal P}_k^{{\cal T}}$. Here $k$ may not be distinct from $a$ or $b$. Note that ${\cal P}_b^{{\cal T}} - {\cal P}_a^{{\cal T}} = {\cal P}_b^{{\cal T}} -{\cal P}_k^{{\cal T}}$. Writing $\pi_a - \pi_b = (\pi_k - \pi_b) - (\pi_k - \pi_a)$ and using a telescopic sum, we get the following result.
\begin{lemma}\label{telescopiclemma}
For two nodes $a$ and $b$ in the flow network
\begin{align}
\pi_a - \pi_b = \smashoperator[r]{\sum_{(cd) \in {\cal P}_b^{{\cal T}} - {\cal P}_a^{{\cal T}}}} g_{cd}(f_{cd}) - \smashoperator[r]{\sum_{(cd) \in {\cal P}_a^{{\cal T}} - {\cal P}_b^{{\cal T}}}} g_{cd}(f_{cd}) \label{telescopic}
\end{align}
\end{lemma}
Before further analysis, we make the following assumption on probability distributions of different nodal injections as reported in literature \cite{bolognani2013identification,distgridpart2}.

\textbf{Assumption $1$:} Nodal injections at non-reference nodes in the network are independent.

Over short time intervals, this assumption is valid as injections are affected by changes/fluctuations in user behavior that are independent. In a subsequent section, we discuss extensions/scenarios where this assumption is relaxed. We now give the following definition for dependence between random variables that is well-studied in literature \cite{positivequad1,lehmann1966}.

\textbf{Definition $1$ \cite{positivequad1,lehmann1966}:} Two random variables are termed \emph{Positive Quadrant Dependent (PQD)} if their probability distributions satisfies: $\mathbb{P}(X\leq x, Y\leq y) \geq \mathbb{P}(X\leq x)\mathbb{P}(Y \leq y)$ for all $x$ and $y$.

In other words, $X$ and $Y$ are PQD if larger (smaller) values of $X$  are associated probabilistically with larger (smaller) values of $Y$. PQD for probability distributions of random variables can thus be thought as analogous to positive correlation for their second moments. Note that two independent random variables are PQD by definition. Further the following lemma holds:

\begin{lemma}\label{PQDsum}
If $X$ and $Y$ are two independent random variables, then $X$ and $X+Y$ are PQD.
\end{lemma}
The proof is listed in the appendix. Next, we state the following result without proof:
\begin{lemma}\label{monotonic} \cite[Theorem $2.4$]{positivequad1}
Monotonic Functions of PQD random variables are positively correlated.
\end{lemma}
Using this we deduce the following result.
\begin{theorem}\label{covarflow}
Let ${\cal V}_1 \subset {\cal V}_2$ be nonempty sets of nodes in $\cal T$. Let $P_{{\cal V}_1} =\sum_{a \in {\cal V}_1}P_a$ and $P_{{\cal V}_2} =\sum_{a \in {\cal V}_2}P_a$. Then for any two flow functions $g_i$ and $g_j$, $g_i(P_{{\cal V}_1})$ and $g_j(P_{{\cal V}_2})$ are positively correlated.
\end{theorem}
\begin{proof}
$P_{{\cal V}_1}$ and $P_{{\cal V}_2} - P_{{\cal V}_1}$  are independent as ${\cal V}_1$ and ${\cal V}_2 - {\cal V}_1$ are disjoint sets and nodal injections are independent. Using Lemma \ref{PQDsum}, $P_{{\cal V}_1}$ and $P_{{\cal V}_2} =  P_{{\cal V}_1} + (P_{{\cal V}_2} - P_{{\cal V}_1})$ are PQD. The result follows from Lemma \ref{monotonic} as flow functions are monotonic.
\end{proof}
We now analyze trends in second moments of nodal potentials using the preceding result. Denote the variance of potential difference $\pi_a - \pi_b$ as $\phi_{ab}$. Using Eq.~(\ref{covar}) and Lemma \ref{telescopiclemma}, we write $\phi_{ab}$ as follows:
\begin{align}
\phi_{ab} &= \mathbb{E}[\pi_a-\pi_b - (\mu_{\pi_a}-\mu_{\pi_b})]^2\nonumber\\
&= \smashoperator[lr]{\sum_{(jk),(st) \in {\cal P}_b^{{\cal T}} - {\cal P}_a^{{\cal T}}}}\Omega_{g_{jk}(f_{jk})g_{st}(f_{st})}+ \smashoperator[r]{\sum_{(jk),(st) \in {\cal P}_a^{{\cal T}} - {\cal P}_b^{{\cal T}}}}\Omega_{g_{jk}(f_{jk})g_{st}(f_{st})} \nonumber\\
&~-  2\smashoperator[lr]{\sum_{(jk) \in {\cal P}_b^{{\cal T}} - {\cal P}_a^{{\cal T}}, (st)\in {\cal P}_a^{{\cal T}} - {\cal P}_b^{{\cal T}}}}\Omega_{g_{jk}(f_{jk})g_{st}(f_{st})} \label{phicovar1}
\end{align}
If $a$ is a descendant of $b$, ${\cal P}_b^{{\cal T}} \subset {\cal P}_a^{{\cal T}}$. Thus, ${\cal P}_b^{{\cal T}} - {\cal P}_a^{{\cal T}}$ is empty and  Eq.~(\ref{phicovar1}) reduces to
\begin{align}
\phi_{ab} &= \smashoperator[lr]{\sum_{(jk),(st) \in {\cal P}_a^{{\cal T}} - {\cal P}_b^{{\cal T}}}}\Omega_{g_{jk}(f_{jk})g_{st}(f_{st})}\label{phicovar}
\end{align}

Note that using Eq.~(\ref{flowinv}), we can express flows on the right side of Eq.~(\ref{phicovar}) and Eq.~(\ref{phicovar1}) in terms of the injections at descendant nodes. The following theorem states a key trend that is observed in $\phi_{ab}$ for a radial flow network.

\begin{theorem} \label{theoremcases}
Consider three nodes $a \neq b \neq c$ in the radial flow network such that the path from $a$ to $c$ passes through $b$. The variance of potential differences $\phi$ satisfies $\phi_{ab} < \phi_{ac}$.
\end{theorem}
\begin{proof}
As the path from $a$ to $c$ passes through $b$, there are three possible configurations for nodes $a$, $b$ and $c$.
\begin{enumerate}
\item $a$ is a descendant of node $b$, $b$ is a descendant of $c$ (see Fig.~\ref{fig:item1}).
\item $c$ is a descendant of node $b$, $b$ is a descendant of $a$ (see Fig.~\ref{fig:item3}).
\item $a$ and $c$ are descendants of node $b$ (see Fig.~\ref{fig:item2}).
\end{enumerate}
\begin{figure}[!bt]
\centering
\hspace*{\fill}
\subfigure[]{\includegraphics[width=0.1625\textwidth]{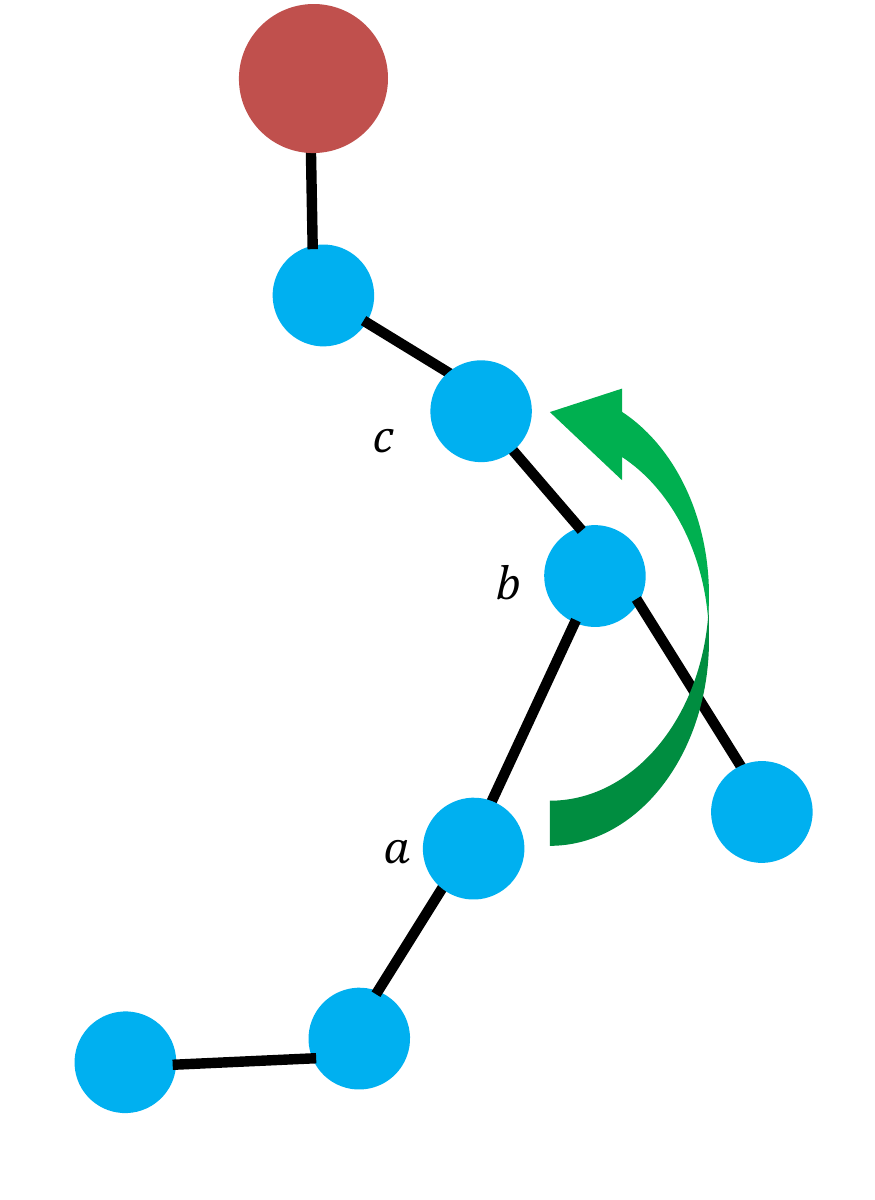}\label{fig:item1}}\hfill
\subfigure[]{\includegraphics[width=0.1625\textwidth]{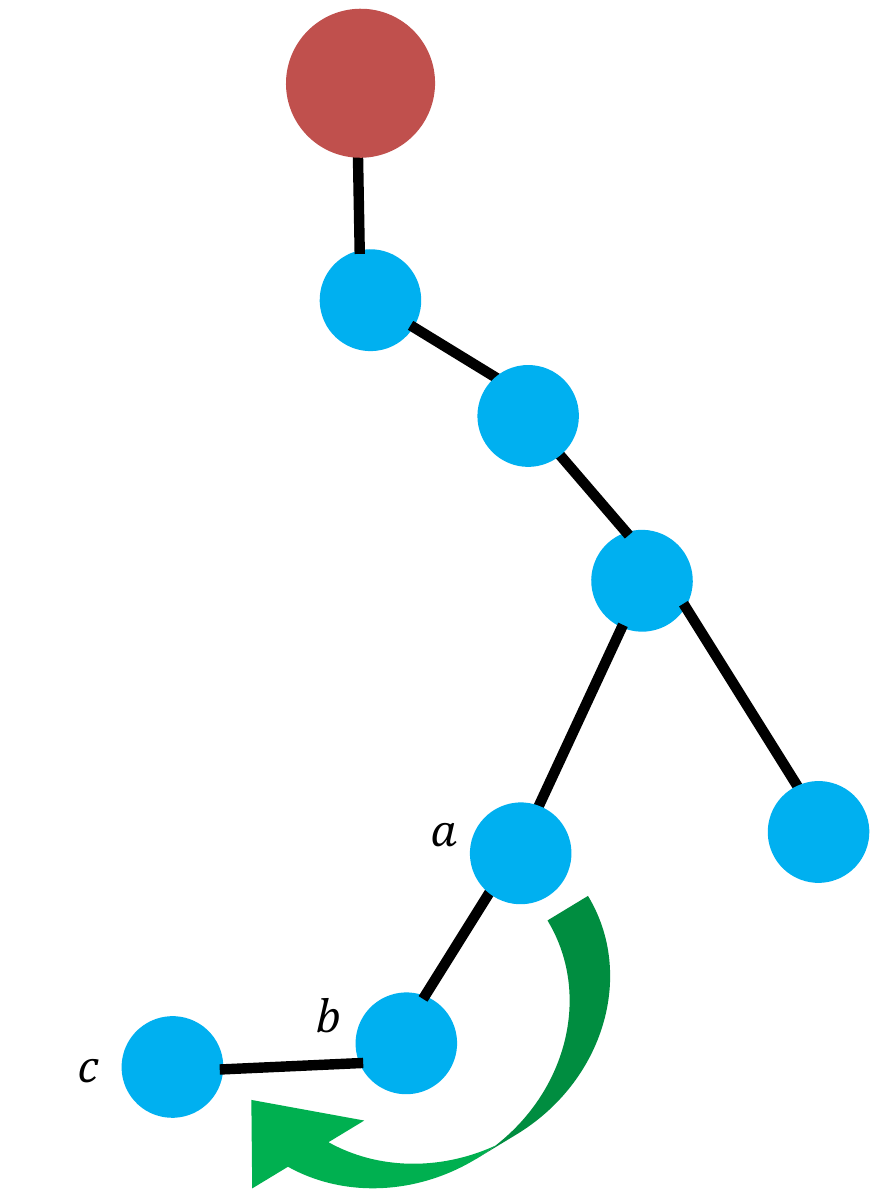}\label{fig:item3}}\hfill
\subfigure[]{\includegraphics[width=0.1559\textwidth]{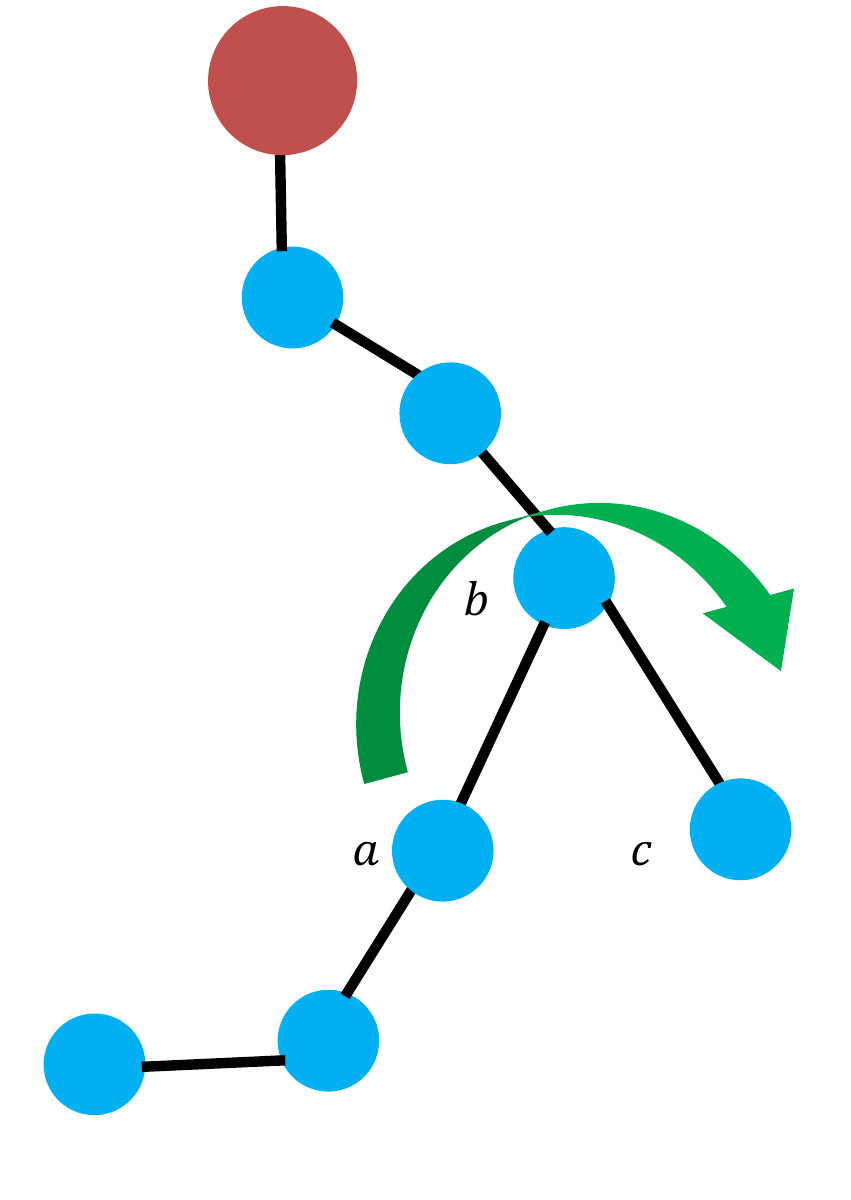}\label{fig:item2}}
\squeezeup
\hspace*{\fill}
\caption{Permissible configurations for nodes $a$, $b$ and $c$ when the path from $a$ to $c$ passes through $b$.
\label{fig:item}}
\end{figure}
To prove the theorem, we consider each case separately and prove the inequality $\phi_{ab} < \phi_{ac}$.

\textbf{Case $1$}: Note that ${\cal P}_c^{{\cal T}} \subset {\cal P}_b^{{\cal T}}\subset {\cal P}_a^{{\cal T}}$. Thus,
$${\cal P}_a^{{\cal T}}- {\cal P}_b ^{{\cal T}} \subset {\cal P}_a^{{\cal T}}- {\cal P}_c ^{{\cal T}}$$. Using this in expression for $\phi_{ab}$ and $\phi_{ac}$ in Eq.~(\ref{phicovar}) gives
\begin{align}
\phi_{ac} - \phi_{ab} = \smashoperator[lr]{\sum_{(jk),(st) \in {\cal P}_b^{{\cal T}} - {\cal P}_c^{{\cal T}}}}\Omega_{g_{jk}(f_{jk})g_{st}(f_{st})} + 2\smashoperator[lr]{\sum_{\substack{(jk) \in {\cal P}_a^{{\cal T}} - {\cal P}_b^{{\cal T}}\\ (st)\in {\cal P}_b^{{\cal T}} - {\cal P}_c^{{\cal T}}}}} \Omega_{g_{jk}(f_{jk})g_{st}(f_{st})} \label{case1formula}
\end{align}
Using Eq.~(\ref{flowinv}), $f_{jk} = \sum_{r \in D_j^{\cal T}} P_r$ and $f_{st} = \sum_{r \in D_s^{\cal T}} P_r$. Note that for any two edges (jk) and (st) in ${\cal P}_a^{{\cal T}} - {\cal P}_c^{{\cal T}}$, $D_j^{\cal T} \subset D_s^{\cal T}$  or $D_s^{\cal T} \subset D_j^{\cal T}$ depending on which node is topologically nearer to $a$. If $(jk) \in {\cal P}_a^{{\cal T}} - {\cal P}_b^{{\cal T}}$, $(st) \in {\cal P}_b^{{\cal T}} - {\cal P}_c^{{\cal T}}$ then $D_s^{\cal T} \subset D_j^{\cal T}$. In either case, using Theorem \ref{covarflow}, we have
$\Omega_{g_{jk}(f_{jk})g_{st}(f_{st})} > 0$. Thus, all terms in Eq.~(\ref{case1formula}) are positive. Thus  $\phi_{ab} <\phi_{ac}$.

\textbf{Case $2$}: In this case, ${\cal P}_a^{{\cal T}} \subset {\cal P}_b^{{\cal T}}\subset {\cal P}_c^{{\cal T}}$. Following the analysis of Case $1$, the expression of $\phi_{ac} - \phi_{ab}$ here becomes
$$\phi_{ac} - \phi_{ab} = \smashoperator[lr]{\sum_{(jk),(st) \in {\cal P}_c^{{\cal T}} - {\cal P}_b^{{\cal T}}}}\Omega_{g_{jk}(f_{jk})g_{st}(f_{st})} + 2\smashoperator[lr]{\sum_{\substack{(jk) \in {\cal P}_c^{{\cal T}} - {\cal P}_b^{{\cal T}}\\ (st)\in {\cal P}_b^{{\cal T}} - {\cal P}_a^{{\cal T}}}}} \Omega_{g_{jk}(f_{jk})g_{st}(f_{st})} $$
Using the same logic as Case $1$, all covariance terms are positive valued and hence $\phi_{ab}< \phi_{ac}$.

\textbf{Case $3$}: From Fig.~\ref{fig:item2} it is clear that the common edges on paths from nodes $a$ and $c$ to the reference node are the ones on the path from node $b$ to the reference node. Thus ${\cal P}_a^{{\cal T}}-{\cal P}_b^{{\cal T}} = {\cal P}_a^{{\cal T}}-{\cal P}_c^{{\cal T}}$ and ${\cal P}_c^{{\cal T}}-{\cal P}_a^{{\cal T}} = {\cal P}_c^{{\cal T}} - {\cal P}_b^{{\cal T}}$. Further, for any edge $(jk)$ in ${\cal P}_a^{{\cal T}}-{\cal P}_b^{{\cal T}}$ and $(st)$ in ${\cal P}_c^{{\cal T}}-{\cal P}_b^{{\cal T}}$, their respective descendant sets $D_j^{\cal T}$ and $D_s^{\cal T}$ are disjoint. By Eq.~(\ref{flowinv}), flows $f_{jk}$ and $f_{st}$ are independent. The expression for $\phi_{ac}$ using Eq.~(\ref{phicovar1}) reduces to
\squeezeup
\begin{align}
\phi_{ac} &= \smashoperator[lr]{\sum_{(jk),(st) \in {\cal P}_a^{{\cal T}} - {\cal P}_b^{{\cal T}}}}\Omega_{g_{jk}(f_{jk})g_{st}(f_{st})} + \smashoperator[lr]{\sum_{(jk),(st) \in {\cal P}_c^{{\cal T}} - {\cal P}_b^{{\cal T}}}}\Omega_{g_{jk}(f_{jk})g_{st}(f_{st})}=\phi_{ab} + \phi_{cb}\label{case3formula}
\end{align}
where Eq.~(\ref{case3formula}) follows from Eq.~(\ref{phicovar}). As $\phi_{cb} > 0$ , $\phi_{ab} < \phi_{ac}$.
Thus, the statement holds as  it is true for all configurations of $a,b$ and $c$ in the network.
\end{proof}
In the next section, we use Theorem \ref{theoremcases} to design our structure learning algorithm.

\section{Structure Learning with Full Observation}
\label{sec:algo1}
The following theorem follows naturally from Theorem \ref{theoremcases}
\begin{theorem}\label{main}
The set of operational edges in the radial flow network $\cal T$ is given by the minimum spanning tree for the loopy graph of all permissible edges where each permissible edge $(ab)$ is given weight $\phi_{ab} = \mathbb{E}[\pi_a-\pi_b-(\mu_{\pi_a}-\mu_{\pi_b})]^2$.
\end{theorem}
\begin{proof}
For each node $a$, the path to all nodes in the operational tree passes through one of its nearest neighbors (parent and children nodes). Using Theorem \ref{theoremcases}, weight $\phi_{ab}$ is, thus, minimum at the nearest neighbors in the tree. The spanning tree constructed using the operational edges thus has the minimum weight among all spanning trees formed from the set of permissible edges.
\end{proof}

\textbf{Algorithm $1$:} The algorithm for constructing the operational network is now straight forward. Using measurements for nodal potentials, permissible edges $(ab)$ are given weights $\phi_{ab}$ and a spanning tree is constructed greedily by picking edges in the increasing order of their edge weights, while avoiding cycles. This is known as Krushkal's algorithm  \cite{kruskal1956shortest,Cormen2001}. If no information on permissible edges is available, then all potential node pairs are considered as permissible and the spanning tree is constructed from the complete graph (every node pair is connected). Note that no information on flow function or nature of probability distribution for individual nodal injections are necessary in Algorithm $1$.

\begin{algorithm}
\caption{Structure Learning using Potential Measurements}
\textbf{Input:} $m$ potential measurements $\pi$ for all nodes, set of all permissible edges $\cal E$ (optional).\\
\textbf{Output:} Operational Edge set ${\cal E}_{\cal T}$.
\begin{algorithmic}[1]
\State Compute $\phi_{ab} = \mathbb{E}[(\pi_a-\mu_{\pi_a}) -(\pi_b-\mu_{\pi_b})]^2$ for all permissible edges
\State Find minimum weight spanning tree from $\cal E$ with $\phi_{ab}$ as edge weights.
\State ${\cal E}_{\cal T} \gets $ {spanning tree edges}
\end{algorithmic}
\end{algorithm}

\textbf{Algorithm Complexity:} Kruskal's Algorithm learns the minimum spanning tree in quasi-linear time in the number of permissible edges in the system. The computational complexity of learning the operational tree is $O(|{\cal E}|\log |{\cal E|})$ where $\cal E$ is the set of all permissible edges. If no information on set $\cal E$ is available, then the complexity (worst-case) becomes $O(|{\cal V}|^2\log |{\cal V}|)$ which is quasi-quadratic in $|{\cal V}|$, the number of nodes in the network.

We now discuss a few extensions of our algorithm to generalized cases.

\textbf{Extension to Multiple Trees}: Our learning algorithm and analysis can be immediately extended to networks with multiple operational trees. In each tree, one can denote one reference node and compute potentials relative to that. Potentials at multiple trees will be uncorrelated and can be separated into different groups before running Algorithm $1$.

\textbf{Learning Flow Functions/Statistics of Nodal Injection}: Note that in our algorithm, no information on flow functions or nodal consumption statistics is necessary. However, if either one of them is known (flow function or statistics of injection) in addition to potential measurements, the other one can be estimated. To obtain this, Algorithm $1$ is first used to learn the structure of the grid and then Eqs.~(\ref{flow}), (\ref{flowinv}) can be used to recursively estimate the flow function or the injection statistics from the leaves up to the reference node.

\textbf{Learning with Missing Nodes}: This refers to the regime where a section of the nodes are measured and potential measurements for others are not available. To learn the structure in the presence of missing nodes, we need additional information pertaining to permissible flow functions and nodal injection statistics. In that case, a modification of Algorithm $1$ can be proposed, where the available measurements of potentials are used to generate a spanning tree without the missing nodes. At the next level, Eqs.~(\ref{phicovar1}) and Eqs.~(\ref{phicovar}) can be used recursively to identify the presence of missing nodes. We plan to expand on this aspect in a future work.

\section{Experiments}
\label{sec:experiments}
In this section, we discuss the performance of Algorithm $1$ in learning the operational radial structure of flow networks using nodal potential measurements as input. To demonstrate the general nature of our work, we present simulation results on two radial networks: a power distribution grid with linear flow function (Fig.~\ref{fig:powercase}), and a gas transmission grid with quadratic flow function (Fig.~\ref{fig:gascase}). The power distribution grid \cite{testcase2,radialsource} consists of $30$ nodes, while the gas grid \cite{zlotnik2015optimal} consists of $25$ nodes. One node is denoted as the reference node with constant potential.

\begin{figure}[!bt]
\centering
\subfigure[]{\includegraphics[width=0.35\textwidth,height = .18\textwidth]{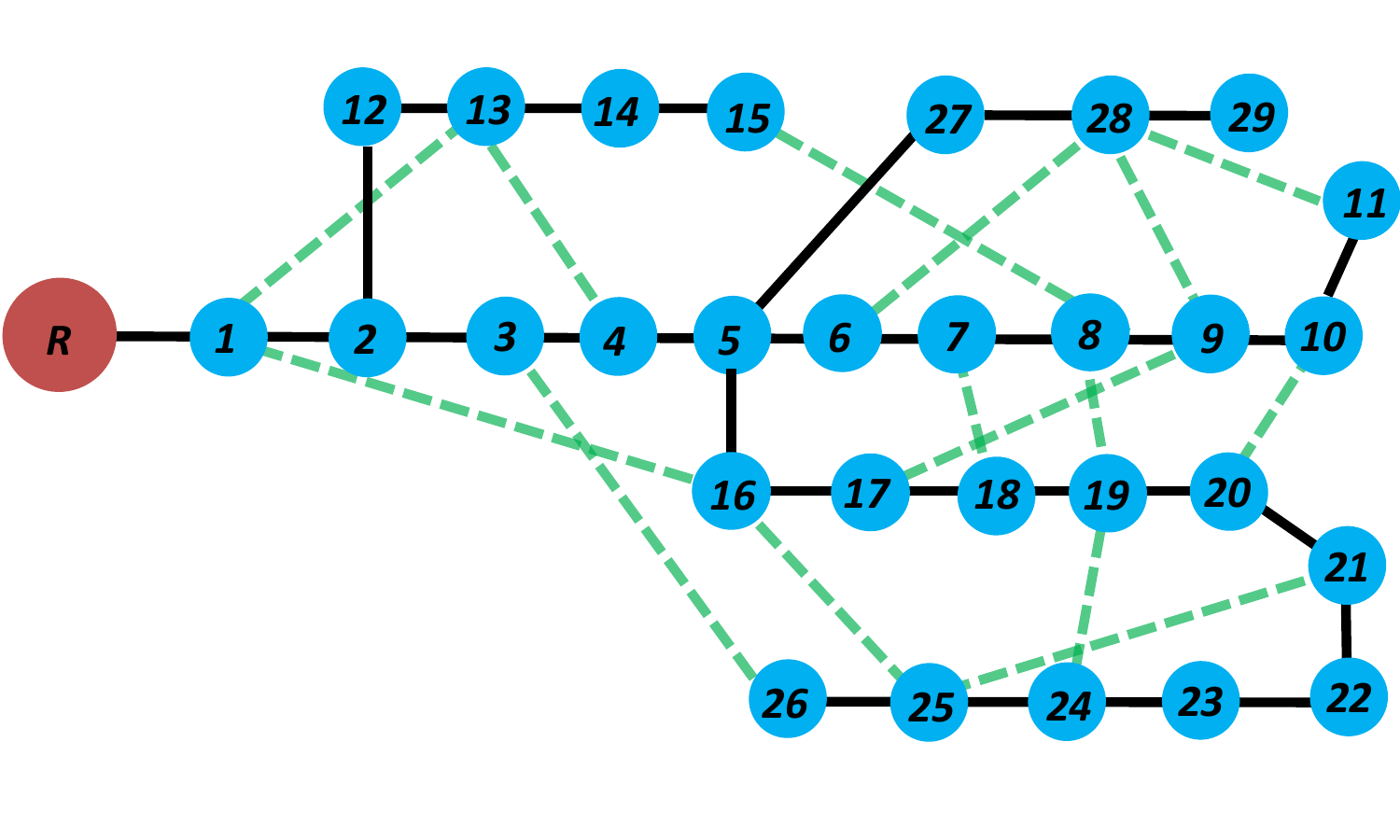}\label{fig:powercase}}
\squeezeup
\subfigure[]{\includegraphics[width=0.34\textwidth,height = .16\textwidth]{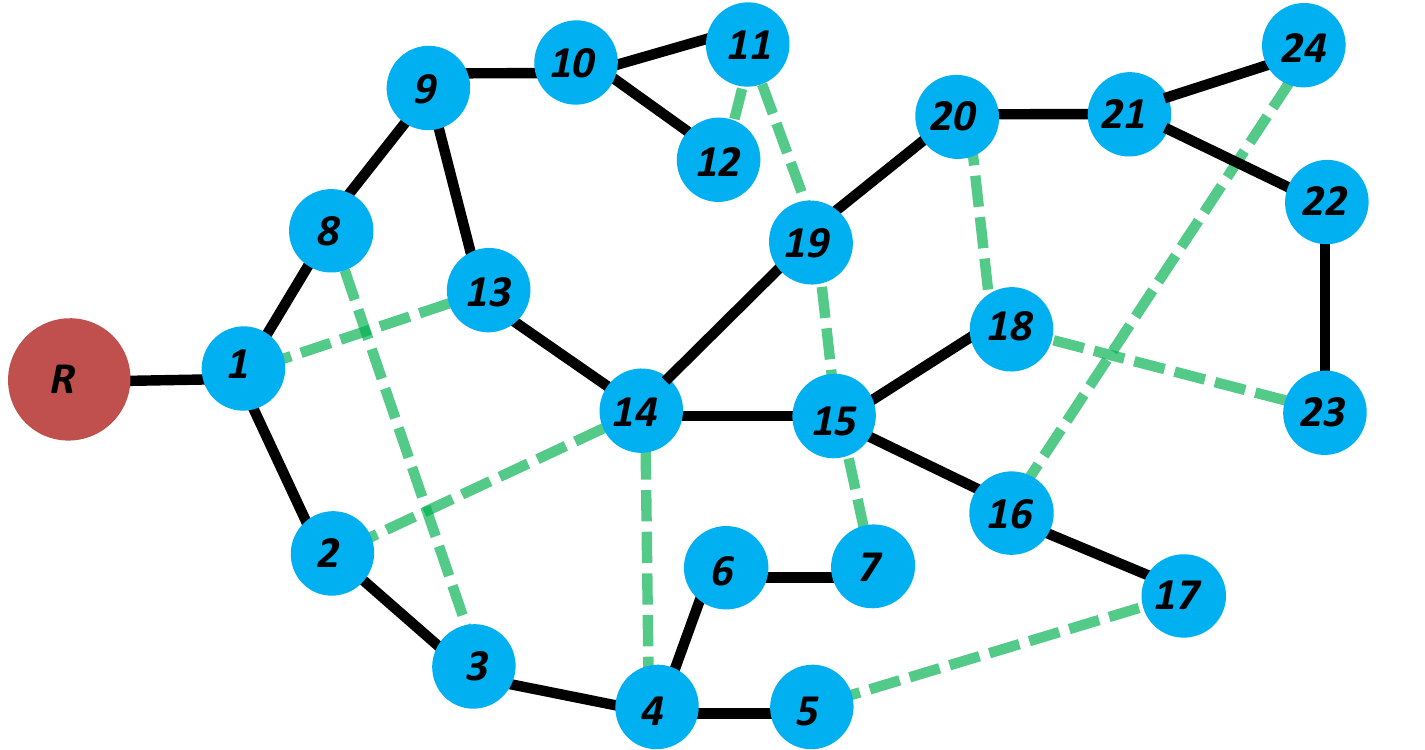}\label{fig:gascase}}
\squeezeup
\caption{Layouts of the grids tested: (a) power distribution grid  \cite{testcase2,radialsource} (b) gas transmission grid \cite{zlotnik2015optimal}. The red and blue circles denote the reference and non-reference nodes respectively in either grid.  Operational edges are colored solid black while some of the fictitious non-operational edges are denoted by dotted green lines.}
\label{fig:case}
\vspace{-2mm}
\end{figure}
To conduct a simulation on either grid, we first generate injection samples at each non-reference node using a uncorrelated multivariate Gaussian distribution. Then flow equations (LinDisFlow Eqs.~(\ref{injpower},\ref{flowpower}) for power grid and Eqs.~(\ref{injgas},\ref{flowgas}) for gas grid) are used to derive input nodal potential measurements (squares of voltage magnitude for power grid and squares of pressures for gas grid). Further, fictitious edges (numbering $30$ for power grid and $25$ for gas grid) are introduced into the loopy set of permissible edges  ${\cal E}$ along with the true operational edges. This is done to observe the performance of structure estimation in Algorithm $1$. The potential measurements and set $\cal E$ are sent as input to Algorithm $1$.

We measure average errors produced in determining the true structure and express them relative to the number of operational edges. To demonstrate the performance for either grid, we plot the trend in average relative errors in Algorithm $1$ versus the number of nodal potential samples available as input. We first consider the case with no measurement noise in Fig.~\ref{fig:plotnonoise}. Notice that the performance is excellent and errors quickly decay to zero for both grids. In fact perfect recovery is observed for samples sizes greater than $100$. All errors here are induced by finite sample sizes that lead to imperfect empirical estimation of $\phi_{ab}$.

Next, we present performance of Algorithm $1$ when potential samples are corrupted with additive noise. We consider the potential samples at each node to suffer from additive Gaussian noise of mean $0$ and variance of value expressed as a fraction of the average variance of nodal potentials. In either grid, we consider three fractions ($8$x$10^{-2}, 5$x$10^{-2}, 10^{-1}$) to represent different levels of noise that are commensurate with noise suffered in off-the shelf measurement devices. Fig.~\ref{fig:plotpowernoise} and Fig.~\ref{fig:plotgasnoise} shows the performance with noise for the power and gas grids respectively. Note that the average fractional errors recorded for either grid go down with increase in the number of samples, though the decay is much slower than in Fig.~\ref{fig:plotnonoise} with no noise. Further, as expected, the errors increase with an increase in the noise variance. It can be observed that the error performance in power grid is significantly better than that in gas networks. This can be explained on the basis of quadratic flow functions in the latter which induce greater errors in empirical approximation of $\phi$ as compared to linear flow functions in the power distribution grid. We plan to theoretically analyze the error performance in detail in future work.

\begin{figure}[!bt]
\centering
\subfigure[]{\includegraphics[width=0.45\textwidth,height = .32\textwidth]{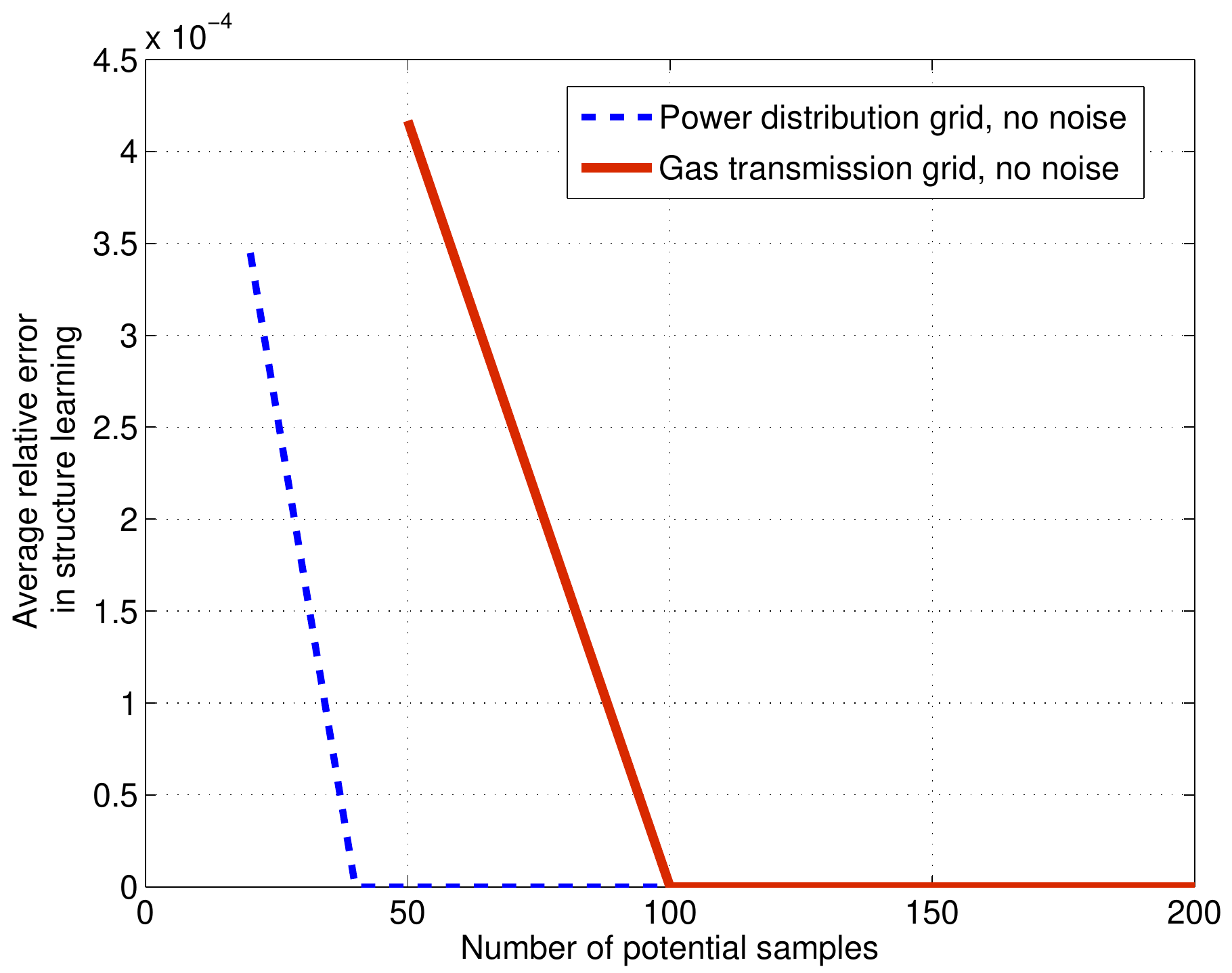}\label{fig:plotnonoise}}
\subfigure[]{\includegraphics[width=0.45\textwidth,height = .32\textwidth]{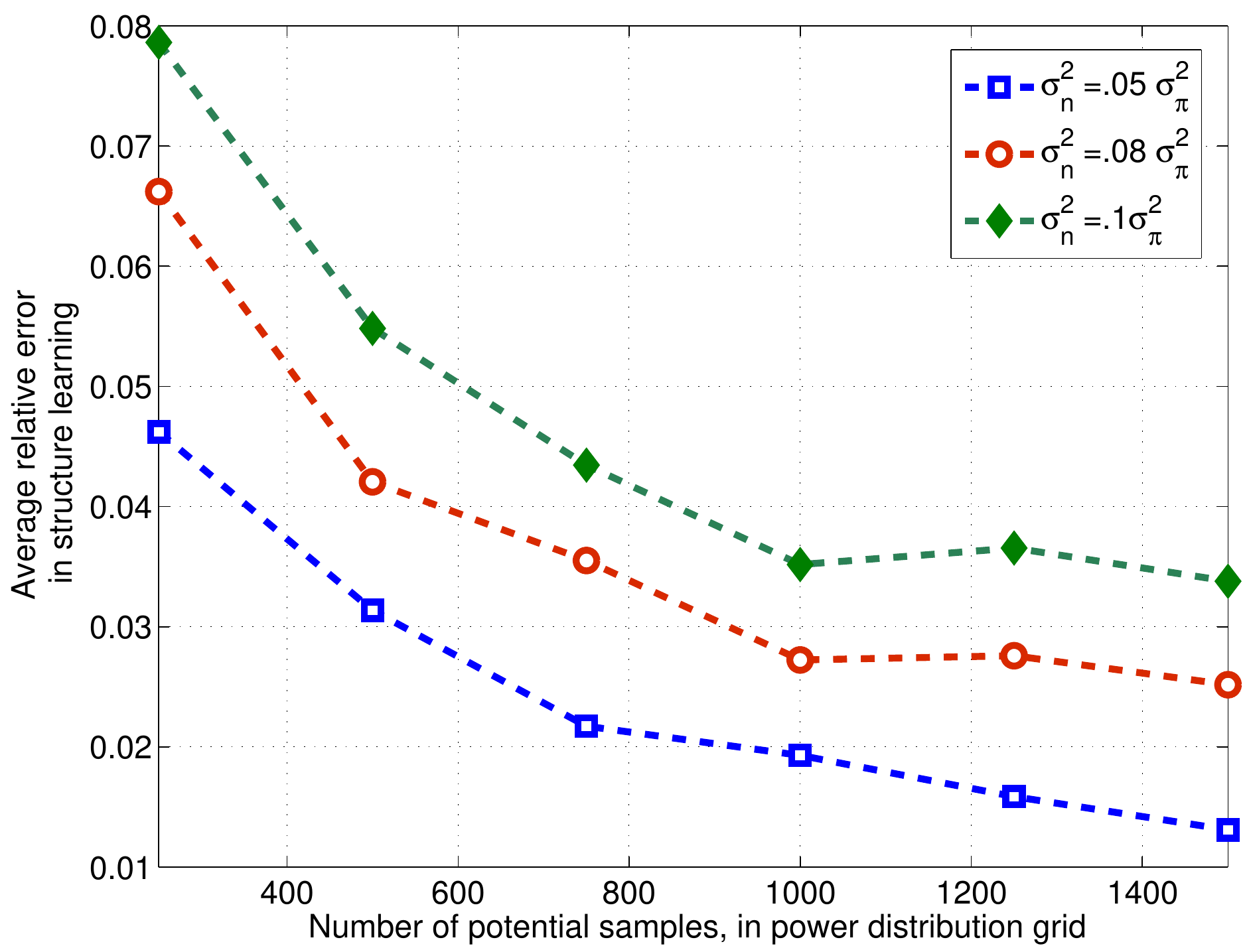}\label{fig:plotpowernoise}}
\subfigure[]{\includegraphics[width=0.45\textwidth,height = .32\textwidth]{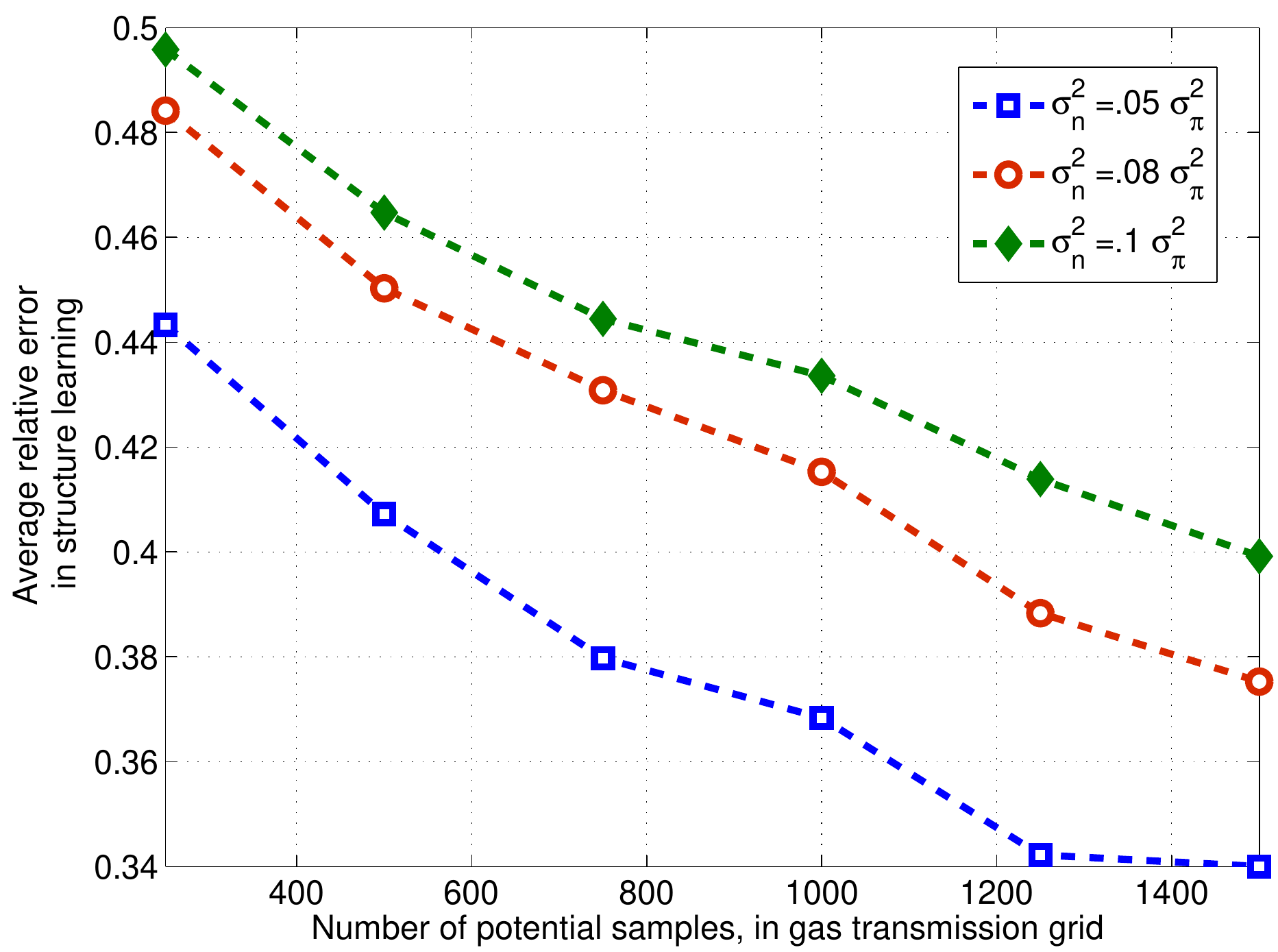}\label{fig:plotgasnoise}}
\squeezeup
\caption{Average fractional errors vs number of samples used in Algorithm $1$ for (a) Power distribution and gas transmission networks with no measurement noise (b) Power distribution network with Gaussian noise (c) Gas transmission network with Gaussian noise. Different noise variances $\sigma_n$ are taken as fractions ($.05,.08,.1$) relative to variance $\sigma_{\pi}$ in nodal potential.
\label{fig:algo1}}
\end{figure}

\section{Conclusions}
\label{sec:conclusions}
Flow networks represent several key infrastructure including power grid, gas grid, water and residential heating networks. Despite the diversity of network traffic, the flow in each network is driven by nodal potentials that are related to the edge flows by a class of nonlinear monotonic flow functions. This paper addresses the problem of estimating the structure of radial flow networks using measurements of nodal potentials. Using properties of positive quadrant dependent functions, we show that the variance of potential differences has provable ordering properties along the network edges. Based on this, a spanning tree based learning algorithm is proposed that can learn the network using only nodal potential statistics. The significant aspect of this algorithm is that it does not require any knowledge of the edge flow functions or specific marginal distributions of nodal injections. This work thus presents the first approach to learning general radial networks with nonlinear flows. The performance of our algorithms are demonstrated through simulations on test radial networks pertaining to a power system and a gas grid. We discuss extensions of our framework, including that with missing/unobserved nodes. Efficient learning of the network structure using potential measurements has application in control and optimization applications as well as in identifying the estimation capability of third parties possessed with limited information. In addition to the extensions mentioned in the paper, potential areas of future work include expanding the learning framework to lossy flows and loopy networks and understanding the sample complexity associated with learning in the presence of noisy measurements.
\appendix
\textbf{Proof of Lemma \ref{PQDsum}}: To prove $X$ and $X+Y$ are PQD, we need to show
\begin{align}
&\mathbb{P}(X+Y \leq b , X \leq a) \geq  \mathbb{P}(X \leq a)\mathbb{P}(X+Y \leq b) \forall a,b\nonumber\\
\Rightarrow~& \mathbb{P}(X+Y \leq b | X \leq a) \geq \mathbb{P}(X+Y \leq b)\label{condinq}
\end{align}
Since we are dealing with physical random variables (power injection etc.), we assume that their probability distribution functions exist. Let $\rho_{X}$ and $\rho_{Y}$ denote the probability distribution functions (p.d.f.s) for $X$ and $Y$ respectively. The p.d.f. for $X+Y$ conditioned on $X\leq a$ is given by:
\begin{align}
&\rho_{X+Y|X\leq a}(z) = \int_{-\infty}^{a}\rho_{X+Y|X=x}(z)\rho_{X|X\leq a}(x)dx \label{chainrule}\\
&~~~~~~~~~~~~~~~~~~~~~~=\frac{\int_{-\infty}^{a}\rho_{Y}(z-x)\rho_{X}(x)dx}{\int_{-\infty}^{a}\rho_{X}(x)dx}\label{ind}\\
\Rightarrow~&\mathbb{P}(X+Y \leq b | X \leq a)= \frac{\int_{-\infty}^{b}\int_{-\infty}^{a}\rho_{Y}(z-x)\rho_{X}(x)dxdz}{\int_{-\infty}^{a}\rho_{X}(x)dx}\nonumber\\
&~~~~~~~~~~~~~~~~~~~~~~= \frac{\int_{-\infty}^{a}\mathbb{P}(Y \leq b-x)\rho_{X}(x)dx}{\int_{-\infty}^{a}\rho_{X}(x)dx}\label{change}
\end{align}
Here, Eq.~(\ref{chainrule}) follows from the chain rule of conditional probability. Eq.~(\ref{ind}) uses the fact that the p.d.f. for $X+Y$ conditioned on $X =a$ is given by $\rho_{X+Y|X=a}(z) = \rho_{Y}(z-a)$ as $X$ and $Y$ are independent. Eq.~(\ref{change}) follows from changing the order of variables $x$ and $z$ under the integration. The right hand side of Eq.~(\ref{change}) represents the weighted average of $\mathbb{P}(Y \leq b-x)$ with weight $\frac{\rho_{X}(x)}{\int_{-\infty}^{a}\rho_{X}(x)dx}$ in $(\infty,a]$ and $0$ otherwise. The derivative of $\mathbb{P}(X+Y \leq b | X \leq a)$ with $a$ is non-positive as shown below:
\begin{align}
&\frac{d}{da}\frac{\int_{-\infty}^{a}\mathbb{P}(Y \leq b-x)\rho_{X}(x)dx}{\int_{-\infty}^{a}\rho_{X}(x)dx} \nonumber\\
&\propto  \mathbb{P}(X \leq a)\mathbb{P}(Y \leq b-a) - \int_{-\infty}^{a}\mathbb{P}(Y \leq b-x)\rho_{X}(x)dx\nonumber\\
&\leq \mathbb{P}(X \leq a)\mathbb{P}(Y \leq b-a) - \mathbb{P}(Y \leq b-a)\int_{-\infty}^{a}\rho_{X}(x)dx \leq 0
\end{align}
The inequality holds as $\mathbb{P}(Y \leq b-x)$ is a decreasing function of $x$.

Thus $\mathbb{P}(X+Y \leq b | X \leq a)$ is non-increasing in $a$ and hence proved that
$$\mathbb{P}(X+Y \leq b | X \leq a) \leq \lim_{a \rightarrow \infty}\mathbb{P}(X+Y \leq b | X \leq a) = \mathbb{P}(X+Y \leq b).$$
\section*{Acknowledgment}
The authors thank S. Misra and A. Zlotnik at Los Alamos National Laboratory for providing information regarding the test gas network used for simulations in the paper.
\bibliographystyle{IEEETran}
\bibliography{../Bib/FIDVR,../Bib/SmartGrid,../Bib/voltage,../Bib/trees}
\end{document}